\documentclass[a4paper, 12pt]{amsart}

\usepackage[inner=30mm, outer=30mm, textheight=245mm]{geometry}

\usepackage{url}
\usepackage{graphicx}

\theoremstyle{plain}
\newtheorem{theorem}{Theorem}[section]

\newtheorem{lemma}[theorem]{Lemma}

\newtheorem{remark}[theorem]{Remark}
\newtheorem{question}[theorem]{Question}

\usepackage[all]{xy}

\title[Geometric triangulations of the figure eight knot complement]{Constructing infinitely many geometric triangulations of the figure eight knot complement}
\author{Blake Dadd}
\author{Aochen Duan}
\address{School of Mathematics and Statistics, University of Melbourne, Parkville, VIC 3010, Australia}
\email{b.dadd@student.unimelb.edu.au,  aochend@student.unimelb.edu.au}
\begin{document}
\maketitle

\section{Introduction}\label{sect:intro}
A triangulation of a 3-manifold is a combinatorial description of that manifold as a set of tetrahedra with faces identified in pairs. 
Every 3-manifold admits infinitely many triangulations, but it is desirable to find triangulations that support additional structure.  
Cusped hyperbolic 3-manifolds (non-compact 3-manifolds with finite volume) form a diverse family of 3-manifolds and the appropriate combinatorial decomposition for them is into ideal tetrahedra, which are tetrahedra with their vertices removed. 
It is conjectured that every cusped hyperbolic 3-manifold admits an ideal triangulation such that each tetrahedron can be isometrically embedded in hyperbolic space $\mathbb{H}^3$  as a geodesic tetrahedron with all four vertices on the sphere at infinity and all dihedral angles positive.
Such a triangulation is called \emph{geometric} and gives a hyperbolic metric on the underlying manifold.
Once such a geometric triangulation is found we could also ask how many such triangulations exist. This paper considers that question for one of the simplest hyperbolic 3-manifolds: the figure eight knot complement. The main result of the investigation is the following theorem.

\begin{theorem}\label{thm:mainthm}
There are infinitely many geometric ideal triangulations of the figure eight knot complement, connected by Pachner 2-3 moves. 
\end{theorem}

As far as we know, this is the first construction of infinitely many geometric triangulations of a cusped hyperbolic 3-manifold.
The proof of this theorem will be deferred to Section ~\ref{sect:thmproof}. This result is made more surprising when contrasted with the behaviour observed in a very similar manifold: the figure eight sister. The approach used to prove Theorem ~\ref{thm:mainthm} did not simply generalise for the figure eight sister, and it is unknown if there are infinitely many geometric triangulations for the figure eight sister. The details of the comparison of the results for the figure eight knot complement and the figure eight sister are contained in Section ~\ref{sect:sister}.

\subsection{Acknowledgements}
 Both authors were supported by the Vacation Scholarship Program in the School of Mathematics and Statistics at The University of Melbourne. The authors would like to thank Craig Hodgson and Neil Hoffman for their supervision and advice.

\section{Background}\label{sect:background}

For background we refer the reader to Thurston's notes~\cite{ThurstonNotes}, especially Chapter 4 which discusses the gluing consistency equations. The hyperbolic geometry of an ideal tetrahedron can be described by associating a complex parameter in 
$\mathbb{C} \setminus \{0,1\}$  to each pair of opposite edges. 
In a labelled triangulation, the associated edge parameters to tetrahedron $A$ on edges $01$, $02$ and $03$  will be denoted $z_A$, $z'_A$ and $z''_A$ respectively (and similarly for tetrahedra $B$, etc.) as in Figure ~\ref{fig:LabelTet}. These complex parameters encode, through the principal value of their complex arguments, the dihedral angles of the tetrahedron. The relationship between the tetrahedral parameters in Figure ~\ref{fig:LabelTet} is dependent on a choice of orientation for the tetrahedra and the complex tetrahedral parameters encode the dihedral angles only when they have positive imaginary parts. 

The question of whether or not an ideal triangulation consisting of such ideal hyperbolic tetrahedra will produce a consistent cusped hyperbolic 3-manifold structure was answered by Thurston ~\cite{ThurstonNotes}. For consistency around an edge class in the triangulation it is required that the product of all complex parameters associated to the edges in the edge class 
is equal to $1$ and that the sum of their principal arguments (the dihedral angles) is equal to $2\pi$. There is also a completeness condition: 
the resulting hyperbolic 3-manifold will be complete if and only if there exist horospherical cusp cross sections which form a closed Euclidean 2-manifold under the action of the face identifications. 

We say that an ideal triangulation of a cusped hyperbolic 3-manifold is \emph{geometric} if there exist complex tetrahedral shape parameters which are consistent with these edge and cusp conditions and the orientation condition, that all tetrahedral parameters have positive imaginary part.  It is sometimes possible to find shape parameters for an ideal triangulation which satisfy the edge and cusp criteria but fail the orientation condition. These  triangulations contain either negatively-oriented tetrahedra (which have complex edge parameters with negative imaginary part) or flat tetrahedra (which have real edge parameters from $\mathbb{R}\backslash\left\{0,1\right\}$) and generally do not give a hyperbolic structure on the manifold.

As a starting point of our investigation we take geometric triangulations and study the effects of Pachner 2-3 moves. A \emph{Pachner 2-3 move} replaces a bipyramidal union formed by two tetrahedra that share a common face with three new tetrahedra as shown in Figure ~\ref{fig:PachParameters}. This construction can be performed with respect to any pair of identified faces in the triangulation as long as they are from distinct tetrahedra. The inverse operation of such a move is known as a \emph{Pachner 3-2 move}. 
Figure ~\ref{fig:PachParameters} illustrates a Pachner 2-3 move replacing two tetrahedra with face pairing $A123=B230$ with three tetrahedra with new face pairings $C013=D013$, $D013=E012$ and $E013=C012$. 

We first point out that if the initial triangulation is known to be geometric, it can be easily determined if the resultant triangulation after the Pachner move is geometric. The argument is as follows: the cusp 
condition and the edge 
condition are both immediately satisfied if each of the three ``equatorial'' edges which were in both tetrahedra in the bipyramid has new edge parameter after the move given by the product of the same edge's parameters in each of the two tetrahedron before the Pachner 2-3 move, as shown in Figure ~\ref{fig:PachParameters}. 
Thus the only concern is that the resulting tetrahedra may be flat or negatively-oriented.
The resulting triangulation will be geometric if the resulting edge parameters of the new tetrahedra have positive imaginary parts.
A Pachner 2-3 (or 3-2) move on a geometric triangulation which results in a geometric triangulation is said to be a \emph{geometric Pachner move}. 

The \emph{Pachner graph} of a cusped 3-manifold is a graph describing the space of all ideal triangulations of that manifold. The Pachner graph's vertices are triangulations of the manifold (distinguished up to isomorphic relabelling) and there is an edge between two vertices 
for each Pachner 2-3 move relating the triangulations at the vertices. This graph is well-known to be connected if all triangulations of the manifold have at least two tetrahedra (see for example \cite{Ma1, Ma2, Pi}). In this paper, the connectivity of the subgraph of consisting of geometric triangulations and geometric Pachner moves will be studied.


Throughout this paper, triangulations will be distinguished up to isomorphisms of the triangulation which affect a relabelling of the vertices. Such equivalence classes of triangulations can be uniquely characterised by a text string known as an \emph{isomorphism signature} introduced by Burton ~\cite{BurtonIsoSig}. These isomorphism signatures can be used with computational package Regina ~\cite{BurtonRegina} to recover the triangulation from the string, and the package SnapPy ~\cite{SnapPy} can give numerical approximations of solutions to the consistency equations of Thurston ~\cite{ThurstonNotes}. For convenience, we provide the isomorphism signatures of many of the triangulations discussed in this paper so our work can be replicated. It is important to note, however, that the labelling of those triangulations used by the authors in this paper does generally not agree with the gluing information encoded in the isomorphism signature. The labellings of triangulations were instead chosen to present a simple inductive argument in the proofs and make the paper more readable.

Thurston famously observed that two regular ideal tetrahedra subject to the face identifications described in Table ~\ref{table:Fig8Tet} form a geometric triangulation of the figure eight knot complement. The next section will be devoted to showing this manifold admits infinitely many geometric ideal triangulations. Moreover, the infinite family of geometric triangulations discovered are connected as a subgraph of the Pachner graph.

\begin{figure}
\includegraphics[scale=0.5]{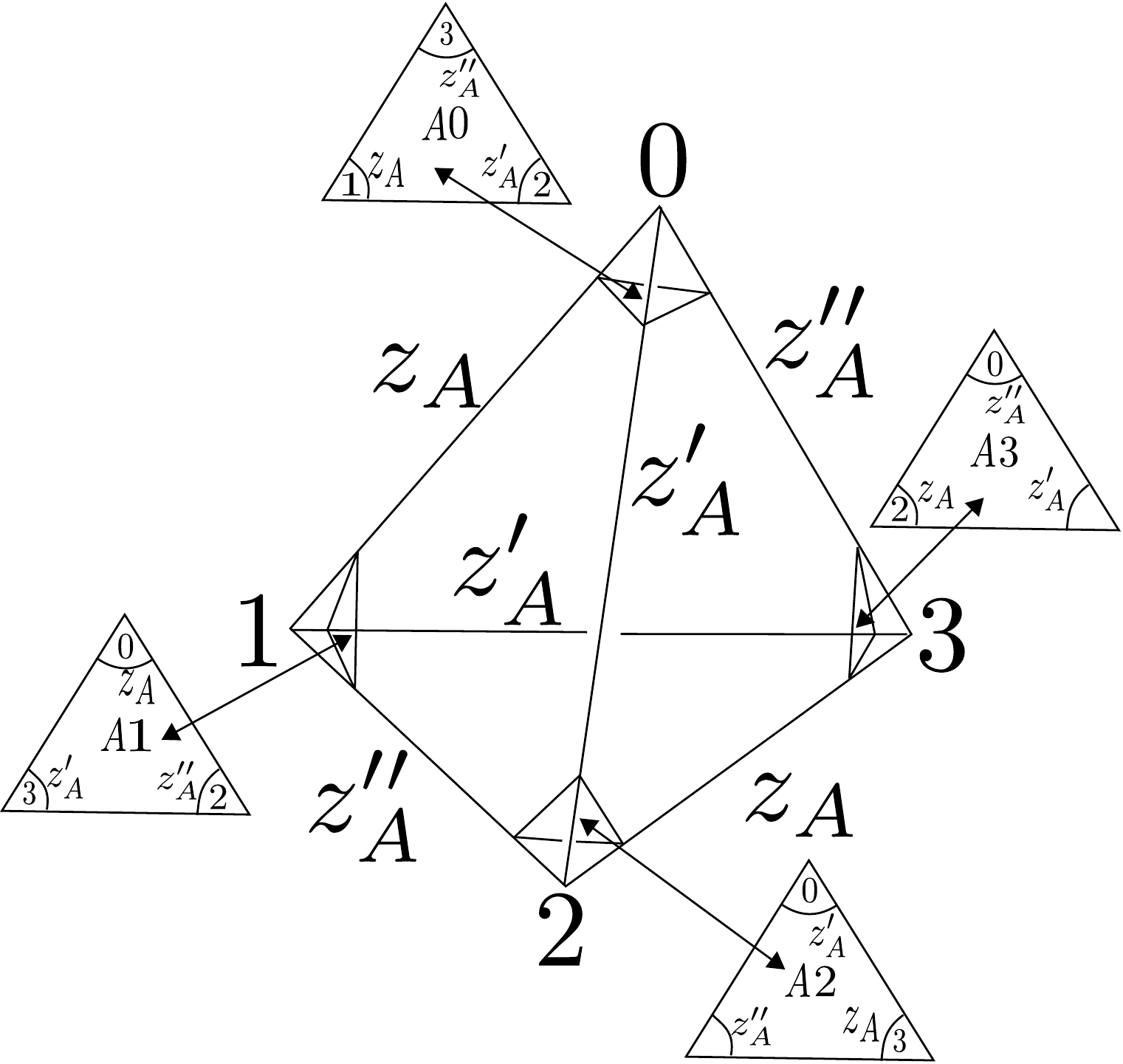}
\caption{\label{fig:LabelTet} The relationship between complex tetrahedral parameters on a single tetrahedron labelled $A$. Here $z_A$, $z'_A$ and $z''_A$ satisfy $z'_A = \frac{1}{1-z_A}$ and $z''_A = \frac{z_A-1}{z_A}$.
Horospherical cross sections at the vertices and their induced labellings are  also shown.
}
\end{figure}

\begin{figure}[h]
\includegraphics[scale=0.6]{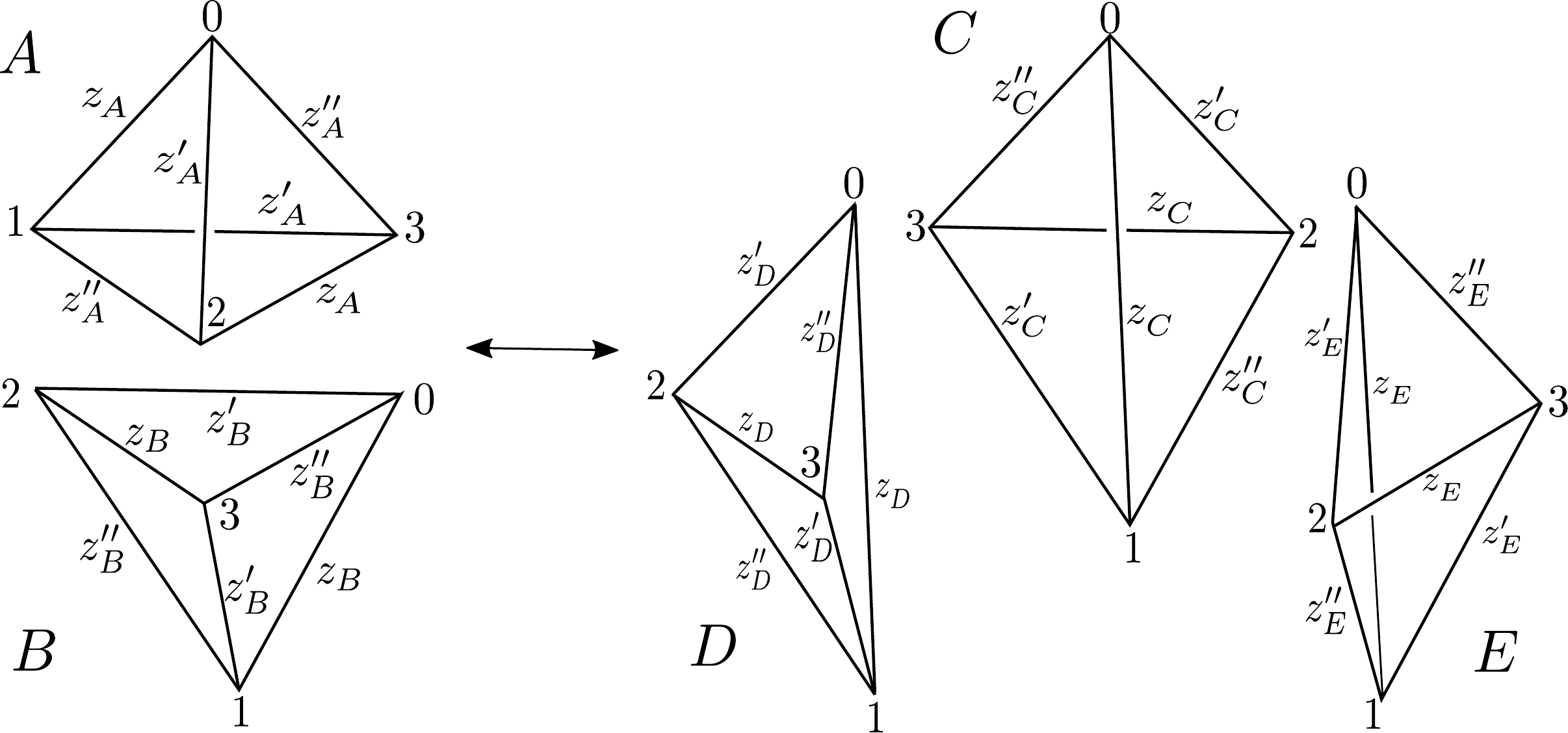}
\caption{\label{fig:PachParameters} How the tetrahedra edge parameters change after a Pachner 2-3 move through face pairing $A123=B230$. To maintain the equatorial angles formed in the bipyramid it is required that $z_C=z'_A z'_B$, $z_D=z''_A z_B$ and 
$z_E=z_A z''_B$.}
\end{figure}

\begin{table}
\begin{tabular}{|c|c|c|c|c|}
\hline
Tetrahedron & 012 & 013 & 023 & 123\\
\hline
\hline
A & B(013) & B(321) &B(021) & B(230)\\
\hline
B & A(032) & A(012) & A(312) & A(310) \\
\hline
\end{tabular}
\caption{\label{table:Fig8Tet} A decomposition of the figure eight knot complement into two tetrahedra.}
\end{table}

\begin{figure}[ht]
\begin{center}
\includegraphics[scale=0.8]{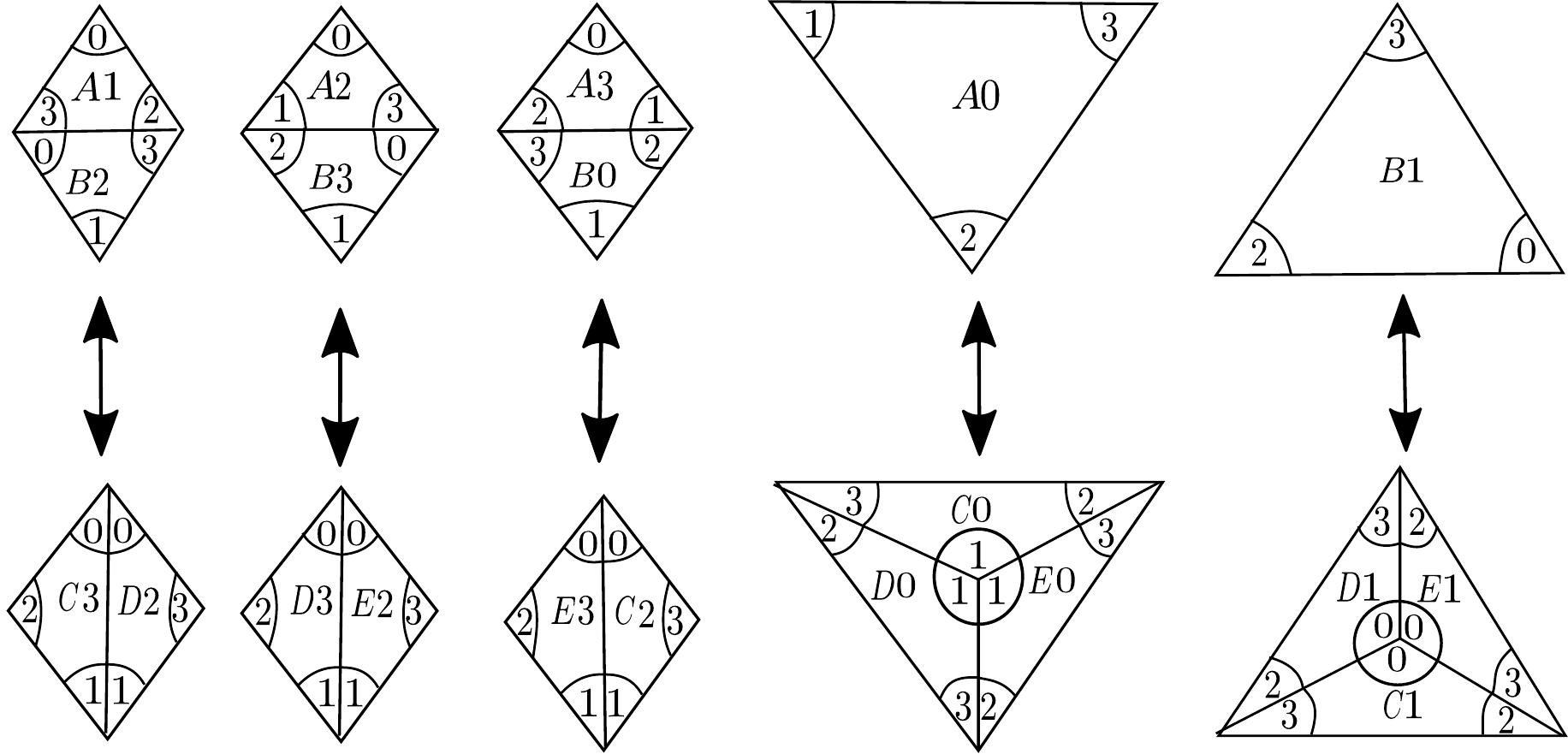}
\end{center}
\caption{\label{fig:PachnerCusps} How the cusp cross sections are changed by a 2-3 Pachner move through the face pairing $A123=B230$. (Compare with Figure ~\ref{fig:LabelTet} and Figure ~\ref{fig:PachParameters}).}
\end{figure}

\section{proof of main theorem}\label{sect:thmproof}
In this section the details of the proof of Theorem ~\ref{thm:mainthm} will be provided. First we define a sequence of labelled triangulations $T_n$, for $n \geq 2$, where $T_2$ is the triangulation of the figure eight knot complement with the labelling given in Table ~\ref{table:Fig8Tet}. The rest of the family is defined recursively where $T_{n+1}$ is obtained by first performing the Pachner 2-3 move through the $A123=B230$ face in $T_{n}$ with the new labelling determined by the convention established in Figure ~\ref{fig:PachParameters}. The resulting triangulation is then relabelled with tetrahedron $E$ relabelled as tetrahedron $A$ and tetrahedron $D$ relabelled as tetrahedron $B$;  this new labelled triangulation is denoted $T_{n+1}$.

The claim in Theorem ~\ref{thm:mainthm} will be proved by showing that all the triangulations $T_n$ for $n \geq 2$ are geometric. As previously mentioned, $T_2$ has a geometric structure, and to justify that the other $T_n$ are also geometric triangulations we use the following lemma. 

\begin{lemma}\label{lem:TechLemma}
Let $T$ be a labelled geometric ideal triangulation of a cusped hyperbolic 3-manifold $M$. Assume that there are two tetrahedra $A$ and $B$ in $T$, such that:
\begin{enumerate}
\item $A$ and $B$ share two faces subject to the following identifications $A123=B230$ and $A012=B013$,
\item $z_A=z_B$ and
\item $\text{Re}(z_A) < 1$, where $\text{Re}(~)$ denotes the real part.
\end{enumerate} 
Then the 2-3 move about the face $A123=B230$ (as in Figure 2) results in a labelled geometric triangulation that has two newly created  tetrahedra $E$ and $D$, that after relabelling by $E\mapsto A$ and $D \mapsto B$ satisfy (1), (2) and (3).  
\end{lemma}

\begin{proof}
A 2-3 move through the face pairing $A123=B230$ is described in Figure ~\ref{fig:PachParameters} and an accompanying depiction of the cusp cross section transformations is in Figure ~\ref{fig:PachnerCusps}. We first note that the new tetrahedra $E$ and $D$ share two faces subject to the identifications $E123=D230$ and $E012=D013$. The existing face identification $A012=B013$ becomes $D023=E312$ or equivalently $E123=D230$. That $E012=D013$ follows from the labelling convention established in Figure ~\ref{fig:PachParameters}. This information is depicted in Figure ~\ref{fig:TechLemmaCuspChanges} which shows the cusp cross sections with the necessary face pairings on the left which transform, as in Figure ~\ref{fig:PachnerCusps}, to produce the cusp cross section diagram on the right where the new face pairings can be recovered.

The new tetrahedral parameters are given by Figure \ref{fig:PachParameters}; they are: 
$$z_C=z'_A z'_B=\frac{1}{1-z_A}\frac{1}{1-z_A}=\frac{1}{(1-z_A)^2}$$
$$z_D=z''_A z_B=\frac{z_A-1}{z_A}z_A=z_A-1$$
$$z_E=z_A z''_B=z_A\frac{z_A-1}{z_A}=z_A-1 .$$ 
%
It follows immediately from these equations that the tetrahedra $D$ and $E$ have the same parameters (specifically $z_D=z_E=z_A-1$), and that $\text{Im}(z_D)>0$, $\text{Im}(z_E)>0$ and $\text{Re}(z_E)<1$. Since $\text{Re}(z_A)<1$ and $\text{Im}(z_A)>0$, 
it follows that $\text{Im}(z_C)>0$. Thus, each of the three new tetrahedra in the resultant labelled triangulation is geometric.  Hence the new labelled triangulation is a geometric triangulation, and relabelling $E$ and $D$ as $A$ and $B$ respectively yields that (1), (2) and (3) are still satisfied.
\end{proof}

\begin{figure}[h]
\includegraphics[scale=0.5]{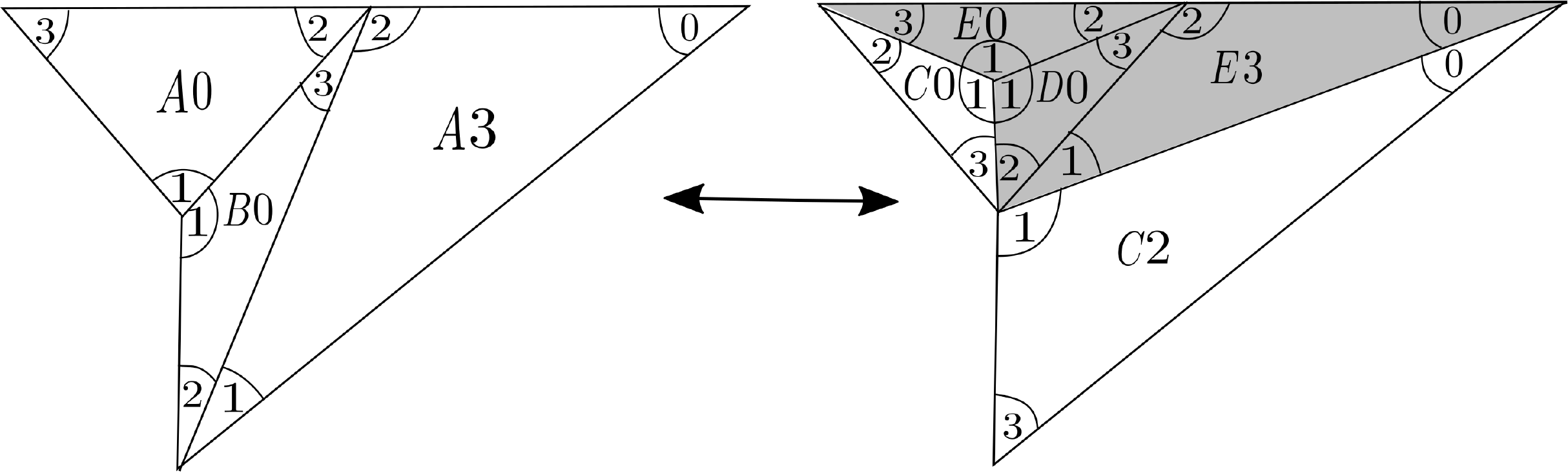}
\caption{\label{fig:TechLemmaCuspChanges} Details of how the tetrahedral cusps are changed by the 2-3 move. The shaded region highlights a self similarity between the triangulations before and after the specific 2-3 move.}
\end{figure}

\begin{proof}[Proof of Theorem \ref{thm:mainthm}]
The above argument shows that if $T_n$ is a geometric labelled triangulation satisfying the three conditions (1)--(3), then $T_{n+1}$ will also be a geometric labelled triangulation satisfying the same three conditions. 
The labelled geometric triangulation $T_2$ of the figure eight knot complement  (Table ~\ref{table:Fig8Tet}), has the face identifications required by Lemma ~\ref{lem:TechLemma}. Since the tetrahedra in this triangulation are regular, all edge parameters are $\frac{1+\sqrt{-3}}{2}$ and thus the figure eight knot complement meets all conditions required in Lemma ~\ref{lem:TechLemma}. Therefore, the described 2-3 move can be performed inductively  to show that $T_n$ is geometric for all $n \geq 2.$
\end{proof}

Thus a geometric triangulation for the figure eight knot complement with $n$ tetrahedra exists for all $n \geq 2$. The triangulations $T_n$ also have quite simple geometries, which stabilise as $n$ increases. In particular, after $k$ Pachner moves in this sequence, the proof of Lemma \ref{lem:TechLemma} shows that
the resultant triangulation $T_{k+2}$ has $2$ tetrahedra (which would be the site of the next 2-3 move) with tetrahedral parameters of 
$\frac{1-2k+\sqrt{-3}}{2}$. The other $k$ tetrahedra 
have distinct shapes with tetrahedral parameters $( \frac{2}{1-2m+\sqrt{-3}})^2$ for each integer value of $m$ from $1$ to $k$, where the tetrahedron with parameter $( \frac{2}{1-2m+\sqrt{-3}})^2$ appeared in the $m\text{th}$ Pachner move.

\begin{remark}
Following Thurston's notes ~\cite[Chapter 7]{ThurstonNotes}, if we define $\text{Vol}(\alpha)$ as the volume of an ideal tetrahedron with $\alpha$ as one of its associated edge parameters, and note that $\text{Vol}\left(\frac{1-2k+\sqrt{-3}}{2}\right)$ decays to $0$ as $n$ becomes large, then we have an alternative way to calculate the volume of the figure eight knot complement: 

The volume of the figure eight knot complement is equal to $$\sum\limits_{m=1}^{\infty} \text{Vol}\left(\left(\frac{2}{1-2m+\sqrt{-3}}\right)^2\right).$$
\end{remark}

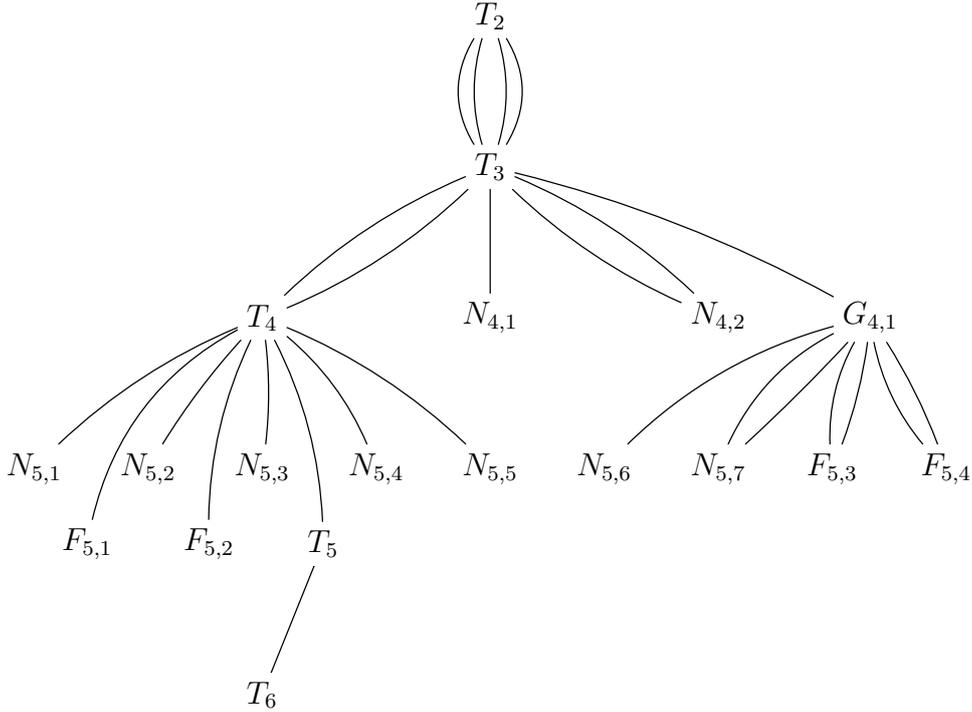
\begin{figure}[h]
\begin{displaymath}
\begin{xy}
(0,20)*+{T_2}="t2"; (0,0)*+{T_3}="t3"; (-30,-20)*+{T_4}="t4";%
(0,-20)*+{N_{4,1}}="n4_1";(30,-20)*+{N_{4,2}}="n4_2";%
(50,-20)*+{G_{4,1}}="g4_1"; (-60,-40)*+{N_{5,1}}="n5_1"; (-45,-40)*+{N_{5,2}}="n5_2"; (-30,-40)*+{N_{5,3}}="n5_3"; (-15,-40)*+{N_{5,4}}="n5_4"; (0,-40)*+{N_{5,5}}="n5_5";%
(-53,-50)*+{F_{5,1}}="d5_1"; (-37,-50)*+{F_{5,2}}="d5_2"; (-22,-50)*+{T_5}="t5"; (15,-40)*+{N_{5,6}}="n5_6"; (30,-40)*+{N_{5,7}}="n5_7"; (45,-40)*+{F_{5,3}}="d5_3";%
 (60,-40)*+{F_{5,4}}="d5_4";%
(-30,-70)*+{T_6}="t6";
{\ar@{-}@/^.5pc/"t2";"t3"};%
{\ar@{-}@/_.5pc/ "t2";"t3"};
{\ar@{-}@/^1pc/"t2";"t3"};%
{\ar@{-}@/_1pc/ "t2";"t3"};
{\ar@{-}@/^.5pc/"t3";"t4"};%
{\ar@{-}@/_.5pc/ "t3";"t4"};
{\ar@{-}"t3";"n4_1"}; 
{\ar@{-}@/^.5pc/"t3";"n4_2"};%
{\ar@{-}@/_.5pc/ "t3";"n4_2"};
{\ar@{-}@/^.5pc/"t3";"g4_1"}; 
{\ar@{-}@/_.5pc/"t4";"n5_1"};
{\ar@{-}@/_.2pc/"t4";"n5_2"};
{\ar@{-}@/^.2pc/"t4";"n5_3"};
{\ar@{-}@/^.5pc/"t4";"n5_4"};
{\ar@{-}@/^.5pc/"t4";"n5_5"};
{\ar@{-}@/_1.2pc/"t4";"d5_1"};
{\ar@{-}@/_.5pc/"t4";"d5_2"};
{\ar@{-}@/^.5pc/"t4";"t5"};
{\ar@{-}@/_.8pc/"g4_1";"n5_6"};
{\ar@{-}@/_.8pc/"g4_1";"n5_7"};
{\ar@{-}@/_.5pc/"g4_1";"d5_3"};
{\ar@{-}@/_.5pc/"g4_1";"d5_4"};
{\ar@{-}@/^.1pc/"g4_1";"n5_7"};
{\ar@{-}@/^.2pc/"g4_1";"d5_3"};
{\ar@{-}@/^.2pc/"g4_1";"d5_4"};
{\ar@{-}"t5";"t6"};
\end{xy}
\end{displaymath}
\caption{\label{fig:MainPachnerTree} Part of the Pachner graph (starting from the minimal triangulation) of the figure eight knot complement, which details the triangulations $T_n$ for low $n$ and nearby triangulations.}
\end{figure}

\begin{table}
\begin{tabular}{|c|c|}
\hline
Triangulation & Isomorphism Signature\\
\hline
\hline
$T_2$ & cPcbbbiht\\
\hline
$T_3$ & dLQbcccdegj\\
\hline
$T_4$ & eLPkbcdddmbkgw\\
\hline
$N_{4,1}$ & eLPkbcdddhgrvv\\
\hline
$N_{4,2}$ & eLPkbcdddmbvgg\\
\hline
$G_{4,1}$ & eLAkbccddhhnqw\\
\hline
$N_{5,1}$ & fLLQcbcdeeeacbdhl\\
\hline
$N_{5,2}$ & fLLQcbeddeehajvid\\
\hline
$N_{5,3}$ & fLLQcbcdeeeacjdqu\\
\hline
$N_{5,4}$ & fLLQcbcdeeelsculi\\
\hline
$N_{5,5}$ & fLLQcbcdeeelsgeap\\
\hline
$F_{5,1}$ & fvPQccdedeefokvss\\
\hline
$F_{5,2}$ & fLLQcbcedeedwufmd\\
\hline
$T_5$ & fLLQcbeddeehhjved\\
\hline
$N_{5,6}$ & fLAMcbccdeemejman\\
\hline
$N_{5,7}$ & fLLQcbeddeehajkiu\\
\hline
$F_{5,3}$ & fLLQcacdeeenkaicr\\
\hline
$F_{5,4}$ & fLLQcaddeeejuaplf\\
\hline
$T_6$ & gLLAQbeddfffhhjvenc\\
\hline
\end{tabular}
\caption{\label{table:IsoSigTable} A list of isomorphism signatures for the triangulations of the figure eight knot complement in Figure ~\ref{fig:MainPachnerTree}.}
\end{table}

The geometric triangulations of the figure eight knot complement in the sequence $T_n$ are by no means an exhaustive list. Figure ~\ref{fig:MainPachnerTree} shows part of the Pachner graph of the figure eight knot complement. Triangulations which are not in the sequence $T_n$ are named as either $G$, $N$ or $F$ depending on whether they are geometric triangulations or contain negatively oriented tetrahedra or flat tetrahedra respectively. The first subscript gives the number of tetrahedra in the triangulation and the second is the order in which they were encountered so that they may be distinguished. The isomorphism signatures for these triangulations are collected in Table ~\ref{table:IsoSigTable}. There is a distinct geometric 4 tetrahedra triangulation which is not in the sequence, though from a search of a neighbourhood of the triangulation $T_2$ in the Pachner graph it appears that the majority of triangulations are not geometric.

\begin{remark}
Neil Hoffman informed the authors of two geometric triangulations of the figure eight knot complement with 5 tetrahedra, neither of which is obtainable from any other triangulation by a Pachner 2-3 move, and which both have the property that every possible Pachner 2-3 move on these triangulations results in a non-geometric triangulation; thus, these triangulations are locally isolated as geometric triangulations. Their isomorphism signatures are fLLQcacdedejbqqww and fLLQccecddehqrwwn.
\end{remark}

\section{The Figure Eight Sister manifold}\label{sect:sister}
This section compares the results obtained for the figure eight knot complement to some work on the figure eight sister manifold (also known as m003). Each of these hyperbolic manifolds has a geometric triangulation consisting of two regular tetrahedra, with differing face identifications, and have the smallest possible volume of orientable cusped hyperbolic 3-manifolds~\cite{CaoMeyerhoff}. These triangulations are \emph{minimal} in the sense that they use the smallest number of tetrahedra possible to produce the manifold; a property that can be easily determined in these simple cases by considering all possible ways of identifying faces of tetrahedra to produce a manifold. However, unlike the figure eight knot complement, only a finite number of distinct isomorphism classes of geometric triangulations could be obtained by performing geometric 2-3 moves starting from the minimal triangulation.

\begin{table}
\begin{tabular}{|c|c|c|c|c|}
\hline
Tetrahedron & 012 & 013 & 023 & 123\\
\hline
\hline
A & B(301) & B(230) &B(021) & B(132)\\
\hline
B & A(032) & A(120) & A(301) & A(132) \\
\hline
\end{tabular}
\caption{\label{table:Fig8Sis} A decomposition of the figure eight knot complement sister manifold into two tetrahedra.}
\end{table}

\begin{theorem}
There are only finitely many geometric triangulations which may be obtained by a sequence of geometric 2-3  moves starting from the minimal triangulation of the figure eight sister manifold.
\end{theorem}

\begin{proof}
The minimal triangulation of the figure eight sister manifold consists of regular ideal tetrahedra subject to the identifications shown in Table ~\ref{table:Fig8Sis}. A 2-3 move through any of the faces results in the same isomorphism class of geometric triangulation, $G'_{3,1}$. From this triangulation, there are five possible sites of 2-3 moves, resulting in 3 isomorphism classes of negatively oriented triangulations and a single geometric triangulation, $G'_{4,1}$. It can now be seen that performing a 2-3 move on any suitable face of this triangulation results in a triangulation with either negatively-oriented or flat tetrahedra. These results are summarised in Figure ~\ref{fig:SisterPachner} (where $G'$, $N'$ and $F'$ represent geometric, negatively-oriented and partially flat triangulations respectively), and the isomorphism signatures of these triangulations are provided in Table ~\ref{table:SisterIsoSigTable}.
\end{proof}

\begin{figure}[h]
\begin{displaymath}
\begin{xy}
(0,20)*+{G'_{2,1}}="g2_1s"; (0,0)*+{G'_{3,1}}="g3_1s"; (-40,-20)*+{G'_{4,1}}="g4_1s";%
 (-15,-20)*+{N'_{4,1}}="n4_1s"; (15,-20)*+{N'_{4,2}}="n4_2s"; (45,-20)*+{N'_{4,3}}="n4_3s";%
(-45,-40)*+{F'_{5,1}}="d5_1s"; (-15,-40)*+{N'_{5,1}}="n5_1s"; (15,-40)*+{N'_{5,2}}="n5_2s"; (45,-40)*+{N'_{5,3}}="n5_3s";
{\ar@{-}@/^.5pc/"g2_1s";"g3_1s"};%
{\ar@{-}@/_.5pc/ "g2_1s";"g3_1s"};
{\ar@{-}@/^1pc/"g2_1s";"g3_1s"};%
{\ar@{-}@/_1pc/ "g2_1s";"g3_1s"};
{\ar@{-}@/_1pc/"g3_1s";"g4_1s"};
{\ar@{-}@/_.5pc/"g3_1s";"n4_1s"};
{\ar@{-}@/^.5pc/"g3_1s";"n4_1s"};
{\ar@{-}@/^.5pc/"g3_1s";"n4_2s"};
{\ar@{-}@/^.5pc/"g3_1s";"n4_3s"};
{\ar@{-}"g4_1s";"d5_1s"};
{\ar@{-}@/^.5pc/"g4_1s";"n5_1s"};
{\ar@{-}@/_.5pc/"g4_1s";"n5_1s"};
{\ar@{-}@/^.5pc/"g4_1s";"n5_2s"};
{\ar@{-}@/^.5pc/"g4_1s";"n5_3s"};
\end{xy}
\end{displaymath}
\caption{\label{fig:SisterPachner} Part of the Pachner graph of the figure eight sister starting from its minimal triangulation $G'_{2,1}$.}
\end{figure}
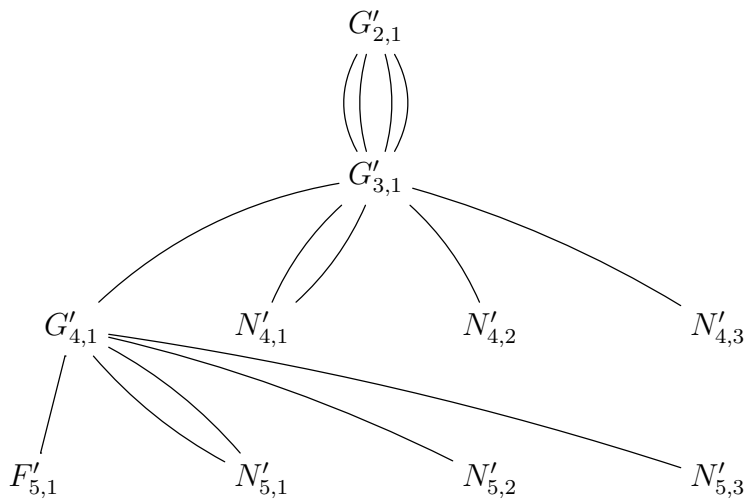

\begin{table}
\begin{tabular}{|c|c|}
\hline
Triangulation & Isomorphism Signature\\
\hline
\hline
$G'_{2,1}$ & cPcbbbdxm\\
\hline
$G'_{3,1}$ & dLQacccjgnb\\
\hline
$G'_{4,1}$ & eLAkaccddngbqw\\
\hline
$N'_{4,1}$ & eLPkbcdddlgjss\\
\hline
$N'_{4,2}$ & eLAkaccddjgnxn\\
\hline
$N'_{4,3}$ & eLAkbbcdddhgaj\\
\hline
$F'_{5,1}$ & fLLQcacdeeenkadnj\\
\hline
$N'_{5,1}$ & fLLQcadedeejaulxk\\
\hline
$N'_{5,2}$ & fLAMcaccdeensnxhj\\
\hline
$N'_{5,3}$ & fLLQccceddetnacwn\\
\hline

\end{tabular}
\caption{\label{table:SisterIsoSigTable} A list of isomorphism signatures for the triangulations of the figure eight sister manifold in Figure ~\ref{fig:SisterPachner}.}
\end{table}

It is worth noting that this does not deny the existence of infinitely many geometric triangulations of the figure eight sister, only an infinite family of geometric triangulations connected by geometric 2-3 Pachner moves and including the minimal triangulation (as was the case for the figure eight knot complement).

\section{Future work}
The figure eight knot complement is the simplest example of a once-punctured torus bundle. A natural question arising from the work in this paper is if there is an extension of Theorem ~\ref{thm:mainthm} to other once-punctured torus bundles. Investigations into the manifold m009 (in SnapPy notation), the next simplest once-punctured torus bundle, revealed that a version of Theorem ~\ref{thm:mainthm} is true for m009. This follows from an argument similar to the one used for the proof for the case of the figure eight knot complement. However the argument is somewhat more complicated in this case as there are two sites in the boundary torus which alternate as the site of the new edge class after each Pachner move, so some extra care needs to be taken. Some further investigations on m023 and other simple once-punctured torus bundles using the computational packages Regina ~\cite{BurtonRegina} and SnapPy ~\cite{SnapPy} suggest it is quite possible that such a family could exist for many once-punctured torus bundles. 

Many of the techniques used and ideas explored throughout this investigation can be used more generally to study cusped hyperbolic 3-manifolds and their triangulations, not just the family of once-punctured torus bundles. It would be of interest to investigate these ideas further and to try to answer the question:

\begin{question}
Which cusped hyperbolic 3-manifolds admit infinitely many geometric ideal triangulations?
\end{question}

%



\end{document}